\documentclass[12pt]{amsart}
\usepackage{stmaryrd}
\usepackage{amssymb}
\usepackage{amsfonts}
\usepackage{amssymb,latexsym}
\usepackage{enumerate}
\usepackage{mathrsfs}
\makeatletter
\@namedef{subjclassname@2010}{%
  \textup{2010} Mathematics Subject Classification}
\makeatother

  [2002/01/19 v2.2g %
    AMS font definitions%
  ]
\DeclareFontFamily{U}{euf}{}
\DeclareFontShape{U}{euf}{m}{n}{%
  <5><6><7><8><9>gen*eufm%
  <10><10.95><12><14.4><17.28><20.74><24.88>eufm10%
  }{}
\DeclareFontShape{U}{euf}{b}{n}{%
  <5><6><7><8><9>gen*eufb%
  <10><10.95><12><14.4><17.28><20.74><24.88>eufb10%
  }{}

  [2002/01/19 v2.2g %
    AMS font definitions%
  ]
\DeclareFontFamily{U}{msb}{}
\DeclareFontShape{U}{msb}{m}{n}{%
  <5><6><7><8><9>gen*msbm%
  <10><10.95><12><14.4><17.28><20.74><24.88>msbm10%
  }{}

  [2002/01/19 v2.2g %
    AMS font definitions%
  ]
\DeclareFontFamily{U}{msa}{}
\DeclareFontShape{U}{msa}{m}{n}{%
  <5><6><7><8><9>gen*msam%
  <10><10.95><12><14.4><17.28><20.74><24.88>msam10%
  }{}

\newtheorem{theorem}{Theorem}[section]
\newtheorem{lemma}[theorem]{Lemma}

\newtheorem{corollary}[theorem]{Corollary}

\theoremstyle{definition}

\newtheorem{remark}[theorem]{Remark}

\numberwithin{equation}{section} 

\begin{document}

\title[hypergeometric Bernoulli polynomials]
{On hypergeometric Bernoulli numbers and polynomials}

\author{Su Hu}

\address{Department of Mathematics, South China University of Technology, Guangzhou, Guangdong 510640, China}
\email{hus04@mails.tsinghua.edu.cn}

\author{Min-Soo Kim}

\address{Center for General Education, Kyungnam University,
7(Woryeong-dong) kyungnamdaehak-ro, Masanhappo-gu, Changwon-si, Gyeongsangnam-do 631-701, Republic of Korea
}
\email{mskim@kyungnam.ac.kr}








\begin{abstract}
In this note, we shall provide several properties of hypergeometric Bernoulli numbers and polynomials, including sums of products identity, differential equations and recurrence formulas.
\end{abstract}

\maketitle

\def\e{z}
\def\l{z}
\def\res{{\rm Res}}


\section{Introduction}

For $N\in\mathbb N,$ Howard \cite{Ho,Ho2} defined the hypergeometric Bernoulli polynomials $B_{N,n}(x)$ by the generating function
\begin{equation}\label{hn-def1-0}
\frac{t^N e^{xt}/N!}{e^t-T_{N-1}(t)}=\sum_{n=0}^{\infty}B_{N,n}(x)\frac{t^{n}}{n!},
\end{equation}
where $T_{N-1}(t)$ is the Taylor polynomial of order $N-1$ for the exponential function.
In particular, when $N=1,$ we recover the classical Bernoulli polynomials, i.e. $B_{1,n}(x)=B_{n}(x)$ (see (\ref{cl-def}) below).

In this note, we shall give an expression for a sum of products of the hypergeometric Bernoulli polynomials  (see Theorem \ref{sum-pro} below),
generalizing some classical results for Bernoulli polynomials. We also present a differential equation and a recurrence relation for these polynomials (see Theorems \ref{appell-dif} and  \ref{appell-rec}).

Recall that, the Bernoulli polynomials is defined by the generating function
\begin{equation}\label{cl-def}
 \frac{te^{xt}}{e^t-1} =\sum_{n=0}^\infty B_{n}(x)\frac{t^n}{n!},
\end{equation}
while the  Bernoulli numbers is defined by $B_n=B_{n}(0)$.

The Bernoulli numbers and polynomials have many applications and satisfy many interesting identities.
The most remarkable  one  is
Euler's sums of products identity
\begin{equation}\label{Ber-id}
\sum_{i=0}^n\binom niB_iB_{n-i}=-nB_{n-1}-(n-1)B_n \quad(n\geq1).
\end{equation}
This identity has been generalized by many authors from different directions (see~\cite{AD,Ch,Di,KMS,KS,Pe1,Pe}).
In particular, Dilcher \cite{Di} provided explicit expressions
for sums of products for arbitrarily many Bernoulli numbers and
polynomials.

Another approach to Bernoulli polynomials is to define them as an Appell sequence with zero mean:
\begin{equation}\label{B0}
B_0(x)=1,
\end{equation}
\begin{equation}\label{Bp}
B'_n(x)=nB_{n-1}(x),
\end{equation}
\begin{equation}\label{Bi}
\int_0^1 B_n(x)dx=\begin{cases}1 & n=0 \\ 0 & n>0. \end{cases}
\end{equation}
Based on the properties of Appell polynomials, He and Ricci~\cite{HR} proved that the Bernoulli polynomials $B_{n}(x)$ satisfy the following differential equations:
\begin{equation}\label{HR}\frac{B_{n}}{n!}y^{(n)}+\frac{B_{n-1}}{(n-1)!}y^{(n-1)}+\cdots+\frac{B_{2}}{2!}y''
-\left(x-\frac 1{2}\right)y'+ny=0.\end{equation}

Howard~\cite{Ho,Ho2} gave a generalization of Bernoulli polynomials by considering the following generating function:
\begin{equation}\label{hn-def}
\frac{t^2 e^{xt}/2}{e^t-1-t}=\sum_{n=0}^{\infty}A_n(x)\frac{t^{n}}{n!}
\end{equation}
and more generally, for all positive integer $N$
\begin{equation}\label{hn-def1}
\frac{t^N e^{xt}/N!}{e^t-T_{N-1}(t)}=\sum_{n=0}^{\infty}B_{N,n}(x)\frac{t^{n}}{n!},
\end{equation}
where $T_{N-1}(t)$ is the Taylor polynomial of order $N-1$ for the exponential function.
For the cases $N=1$ and $N=2,$ (\ref{hn-def1}) reduces to (\ref{cl-def}) and (\ref{hn-def}), respectively.
We see that the polynomials $B_{N,n}(x)$ have rational coefficients.

The polynomials $B_{N,n}(x)$ are named hypergeometric Bernoulli polynomials, while the numbers $B_{N,n}=B_{N,n}(0)$ are named
hypergeometric Bernoulli numbers since the generating function $f(t)=\frac{e^t-T_{N-1}(t)}{t^N/N!}$ can be expressed as ${_1F_1(1,N+1;t)},$
where the confluent hypergeometric function $_1F_1(a,b;t)$ is defined by
$$_1F_1(a,b;t)=\sum_{n=0}^\infty\frac{(a)_n}{(b)_n}\frac{t^n}{n!}$$
and $(a)_n$ is the Pochhammer symbol
$$(a)_n=\begin{cases} a(a+1)\cdots(a+n-1)&(n\geq1) \\
1&(n=0)
\end{cases}$$ (see \cite[p. 2261]{Ka}).

As their classical counterparts, the Bernoulli numbers and polynomials, the generalized Bernoulli numbers and polynomials,
the hypergeometric Bernoulli numbers and polynomials also satisfy many interesting properties
(\cite{zeta,HN,Hassen,Lu,Nguyen,NC}).

Kamano~\cite{Ka}  proved  the following result for sums of products of hypergeometric Bernoulli numbers, which is a generalization of
the works by Euler and Dilcher (see \cite{Di}).

\begin{theorem}[{Kamano \cite[p.~2262, Main Theorem]{Ka}}]\label{Kamano}
Let $N$ and $r$ be positive integers. For any integer $n\geq r-1,$ we have
$$
\begin{aligned}
\sum_{{\substack{i_1,\ldots,i_r\geq0\\i_1+\cdots+i_r=n}}}&\frac{n!}{i_{1}!\cdots i_{r}!}B_{N,i_1}\cdots B_{N,i_r} \\
&=\frac1{N^{r-1}}\sum_{i=0}^{r-1}A_r^{(N)}(i;1+N(r-1)-n)(-1)^i\binom ni i! B_{N,n-i},
\end{aligned}
$$
where $A_r^{(N)}(i;s)\in\mathbb Q[s](0\leq i\leq r-1)$ are polynomials defined by the following recurrence relation:
\begin{equation}\label{thm-eq}
\begin{aligned}
&A_1^{(N)}(0;s)=1 \\
&A_r^{(N)}(i;s)=\frac{s-1}{r-1}A_{r-1}^{(N)}(i;s-N)+A_{r-1}^{(N)}(i-1;s-N+1).
\end{aligned}
\end{equation}
Here $r\geq2$ and $A_{r}^{(N)}(i;s)$ are defined to be zero for $i\leq-1$ and $i\geq r.$
\end{theorem}

As for Bernoulli polynomials, another approach to hypergeometric Bernoulli polynomials is to define them in terms of Appell sequence with zero mean (comparing with (\ref{B0}), (\ref{Bp}) and (\ref{Bi}) above):
\begin{equation}\label{HB0}
B_{N,0}(x)=1,
\end{equation}
\begin{equation}\label{HBp}
B'_{N,n}(x)=nB_{N,n-1}(x),
\end{equation}
\begin{equation}\label{HBi}
\int_0^1 (1-x)^{N-1}B_{N,n}(x)dx=\begin{cases}\frac{1}{N} & n=0 \\ 0 & n>0\end{cases}
\end{equation} (see \cite[p. 768]{HN}).

This paper contains further properties of the hypergeometric Bernoulli numbers and polynomials, including a
generalization of Kamano's result.

\begin{theorem}\label{sum-pro}
Let $N$ and $r$ be positive integers and let $x=x_1+\cdots+x_r.$ For any integer $n\geq r-1,$ we have
$$
\begin{aligned}
\sum_{{\substack{i_1,\ldots,i_r\geq0\\i_1+\cdots+i_r=n}}}&\frac{n!}{i_{1}!\cdots i_{r}!}B_{N,i_1}(x_1)\cdots B_{N,i_r}(x_r) \\
&=\frac1{N^{r-1}}\sum_{i=0}^{r-1}A_r^{(N)}(i,x;1+N(r-1)-n)(-1)^i\binom ni i! B_{N,n-i}(x),
\end{aligned}
$$
where $A_r^{(N)}(i,x;s)\in\mathbb Q[x,s](0\leq i\leq r-1)$ are polynomials defined by the recurrence relation:
\begin{equation}\label{thm-eq0}
\begin{aligned}
&A_1^{(N)}(0,x;s)=1 \\
&A_r^{(N)}(i,x;s)=\frac{s-1}{r-1}A_{r-1}^{(N)}(i,x;s-N)-\frac{x-(r-1)}{r-1}A_{r-1}^{(N)}(i-1,x;s-N+1).
\end{aligned}
\end{equation}
Here $r\geq2$ and $A_{r}^{(N)}(i,x;s)$ are defined to be zero for $i\leq-1$ and $i\geq r.$
\end{theorem}

\begin{remark}
Theorem \ref{Kamano} is the special case if $x=0.$
Nguyen and Cheong \cite{NC} also obtained a similar statement using a result by Nguyen~\cite{Nguyen}.
\end{remark}

\begin{remark}
Let $r=2,3$ in  (\ref{thm-eq0}) to obtain
 \begin{equation}\label{coro-1}
\begin{aligned}
&A_2^{(N)}(0,x;1+N-n)= N-n, \\
&A_2^{(N)}(1,x;1+N-n)= -(x-1), \\
&A_3^{(N)}(0,x;1+2N-n)= \frac12(2N-n)(N-n), \\
&A_3^{(N)}(1,x;1+2N-n)= -\frac12(2N-n)(x-1)-\frac12(x-2)(N-n+1), \\
&A_3^{(N)}(2,x;1+2N-n)= \frac12(x-2)(x-1),
\end{aligned}
\end{equation}
thus
\begin{equation}\label{2-sum}
\begin{aligned}
\sum_{{\substack{ i_1,i_2\geq0\\i_1+i_2=n}}}&
\frac{n!}{i_{1}!i_{2}!}B_{N,i_1}(x_1)B_{N,i_2}(x_2) \\
&=\frac1N(N-n)B_{N,n}(x)+\frac nN(x-1)B_{N,n-1}(x) \\
&\quad(\text{where }x=x_1+x_2 \text{ and } n\geq1),
\end{aligned}
\end{equation}
and
\begin{equation}\label{3-sum}
\begin{aligned}
\sum_{{\substack{i_1,i_2,i_3\geq0\\i_1+i_2+i_3=n}}}&
\frac{n!}{i_{1}!i_{2}!i_{3}!}B_{N,i_1}(x_1)B_{N,i_2}(x_2)B_{N,i_3}(x_3) \\
&=\frac1{2N^2}\biggl[(N-n)(2N-n)B_{N,n}(x) \\
&\quad+n((2N-n)(x-1)+(x-2)(N-n+1))B_{N,n-1}(x) \\
&\quad+n(n-1)(x-1)(x-2)B_{N,n-2}(x)\biggl] \\
&\quad(\text{where }x=x_1+x_2+x_3\text{ and } n\geq2).
\end{aligned}
\end{equation}
Nguyen and Cheong \cite[Example 18]{NC} also obtained (\ref{2-sum}) and (\ref{3-sum}).
If $N=1$, $x=0$ and $r=2$, then (\ref{2-sum}) becomes (\ref{Ber-id}).
\end{remark}

For $N,r\in\mathbb N,$
the higher order hypergeometric Bernoulli polynomials $B_{N,n}^{(r)}(x)$ are defined by the generating function
\begin{equation}\label{hn-defh}
\left(\frac{t^N/N!}{e^t-T_{N-1}(t)}\right)^re^{xt}=\sum_{n=0}^\infty B_{N,n}^{(r)}(x)\frac{t^n}{n!}.
\end{equation}
The higher order hypergeometric Bernoulli numbers are defined by $B_{N,n}^{(r)}=B_{N,n}^{(r)}(0)$ (see \cite{Ka,NC}).
The hypergeometric Bernoulli case corresponds to the special value $r = 1.$
In particular, $B_{N,n}^{(1)}(x)=B_{N,n}(x),$ the hypergeometric Bernoulli polynomials, and
$B_{N,n}^{(1)}(0)=B_{N,n},$ the hypergeometric Bernoulli numbers.

If we put $N=1$ in (\ref{hn-defh}), then we have $B_{1,n}^{(r)}(x)=B_{n}^{(r)}(x)$ the classical higher order Bernoulli polynomials.
The classical higher order Bernoulli numbers is defined by $B_n^{(r)}=B_{n}^{(r)}(0)$
(see \cite{Di,KMS,Lu,Pe}).

Using results established in \cite{HR} for the classical Bernoulli polynomials a differential equation for the higher
order Bernoulli polynomials is derived next. The proof is based on the properties of Appell polynomials \cite{Ap,Sh}.

\begin{theorem}\label{appell-dif}
The higher order hypergeometric Bernoulli polynomials $B_{N,n}^{(r)}(x)$ satisfy the differential equation
$$\frac{B_{N,n}}{n!}y^{(n)}+\frac{B_{N,n-1}}{(n-1)!}y^{(n-1)}+\cdots+\frac{B_{N,2}}{2!}y''
-\left(\frac{x}{rN}-\frac 1{N(N+1)}\right)y'
+\frac{n}{rN}y=0.$$
\end{theorem}

\begin{corollary}[Lu {\cite[Theorem 2.2]{Lu}}]\label{coro-0}
The classical higher order Bernoulli polynomials $B_{n}^{(r)}(x)$ satisfy the differential equation
$$\frac{B_{n}(1)}{n!}y^{(n)}+\frac{B_{n-1}(1)}{(n-1)!}y^{(n-1)}+\cdots+\frac{B_{2}(1)}{2!}y''
-\left(\frac{x}{r}-\frac 1{2}\right)y'
+\frac{n}{r}y=0.$$
\end{corollary}

Theorem \ref{appell-dif} in the case $N = r = 1$ gives the next result.

\begin{corollary}[He and Ricci {\cite[Theorem 2.3]{HR}}]\label{coro-2}
The classical Bernoulli polynomials $B_{n}(x)$ satisfy the differential equation
$$\frac{B_{n}}{n!}y^{(n)}+\frac{B_{n-1}}{(n-1)!}y^{(n-1)}+\cdots+\frac{B_{2}}{2!}y''
-\left(x-\frac 1{2}\right)y'+ny=0.$$
\end{corollary}

A linear recurrence for higher order Bernoulli polynomials generalizing results of Lu \cite{Lu} appears as consequence of the proof.

\begin{theorem}\label{appell-rec}
For $n\in\mathbb N,$ the higher order hypergeometric Bernoulli polynomials $B_{N,n}^{(r)}(x)$ satisfy the recurrence
$$B_{N,n+1}^{(r)}(x)=\left(x-\frac r{N+1}\right)B_{N,n}^{(r)}(x)-rN\sum_{k=0}^{n-1}\binom nk\frac{B_{N,n-k+1}}{n-k+1}B_{N,k}^{(r)}(x).$$
\end{theorem}

The special case $N = 1$ gives the following statement.

\begin{corollary}[Lu {\cite[Theorem 2.1]{Lu}}]\label{coro-l}
For $n\in\mathbb N,$ the higher order Bernoulli polynomials $B_{n}^{(r)}(x)$ satisfy the recurrence
$$B_{n+1}^{(r)}(x)=\left(x-\frac 1{2}r\right)B_{n}^{(r)}(x)-r\sum_{k=0}^{n-1}\binom nk\frac{B_{n-k+1}(1)}{n-k+1}B_{k}^{(r)}(x).$$
\end{corollary}

Letting  $N=r=1$ and replace $n$ by $n+1$ in Theorem \ref{appell-rec}, we obtain the next result.

\begin{corollary}[He and Ricci {\cite[Theorem 2.2]{HR}}]\label{coro-2}
For $n\in\mathbb N,$ the Bernoulli polynomials $B_{n}(x)$ satisfy the recurrence
$$B_{n}(x)=\left(x-\frac 1{2}\right)B_{n-1}(x)-\frac1n\sum_{k=0}^{n-2}\binom nk B_{n-k} B_{k}(x).$$
\end{corollary}

\begin{remark}
The special case $r = 1$ gives some results presented in \cite{Na}.
\end{remark}

\section{Proof of Theorem \ref{sum-pro}}\label{thm-proof}

Introduce the notations
\begin{equation}\label{hn-def-po}
F_{r,N}(t,x)=\left(\frac{t^N/N!}{e^t-T_{N-1}(t)}\right)^{r}e^{xt},
\end{equation}
\begin{equation}\label{hn-def-pol}
F_{r,N}(t)=F_{r,N}(t,0)=\left(\frac{t^N/N!}{e^t-T_{N-1}(t)}\right)^r
\end{equation}
and
\begin{equation}\label{hn-def-nu}
F_N(t)=F_{1,N}(t)=\frac{t^N/N!}{e^t-T_{N-1}(t)}.
\end{equation}

The proof begins with an auxiliary result.

\begin{lemma}\label{lem1}
$\quad$
\begin{enumerate}
\item[\rm (i)] $\frac{\rm d}{{\rm d}t}F_N(t)=\frac Nt F_N(t) -F_N(t)-\frac Nt F_{2,N}(t).$
\item[\rm (ii)] $\frac{\rm d}{{\rm d}t}F_{r,N}(t,x)=\frac {rN}t F_{r,N}(t,x) +(x-r)F_{r,N}(t,x)-\frac {rN}t F_{r+1,N}(t,x).$
\end{enumerate}
\end{lemma}
\begin{proof}
Differentiating (\ref{hn-def-nu}) with respect to $t$, yields
$$\begin{aligned}
\frac{\rm d}{{\rm d}t}F_N(t)&=\frac{Nt^{N-1}/N!(e^t-T_{N-1}(t))-t^N/N!(e^t-T_{N-2}(t))}{(e^t-T_{N-1}(t))^2}\\
&=\frac Nt\frac{t^{N}/N!}{(e^t-T_{N-1}(t))}-\frac{t^N/N!(e^t-T_{N-1}(t)+t^{N-1}/(N-1)!)}{(e^t-T_{N-1}(t))^2}\\
&=\frac Nt F_N(t) -F_N(t)-\frac Nt F_{2,N}(t),
\end{aligned}$$
giving Part (i). Part (ii) follows from Part (i).
\end{proof}

The notations introduced above states that
\begin{equation}
F_{r,N}(t,x)=\sum_{n=0}^\infty B_{N,n}^{(r)}(x)\frac{t^n}{n!}.
\end{equation}
Comparing coefficients of equal powers in Lemma \ref{lem1}(ii) gives the recurrence stated next.

\begin{lemma}\label{lem3}
Let $n\in\mathbb N.$ Then
$$B_{N,n}^{(r+1)}(x)=\frac1N\left(N-\frac nr\right)B_{N,n}^{(r)}(x)+\frac1N\frac nr(x-r)B_{N,n-1}^{(r)}(x).$$
\end{lemma}

Let $\binom{n}{i_1,\ldots,i_r}=\frac{n!}{i_1!\cdots i_r!}$
be the multinomial coefficient. The next result gives $B_{N,n}^{(r)}(x)$ in terms of $B_{N,i}(x).$

\begin{lemma}\label{lem4}
The identity holds for $n,r,N\in\mathbb N$
\begin{equation}\label{m-B-E}
B_{N,n}^{(r)}(x)=\sum_{\substack{i_1+\cdots+i_r=n \\ i_1,\ldots,i_r\geq0}}
\binom{n}{i_1,\ldots,i_r}B_{N,i_1}(x_1)\cdots B_{N,i_r}(x_r)
\end{equation}
with $x=x_1+\cdots+x_r$.
\end{lemma}

\begin{proof}[Proof of the Theorem \ref{sum-pro}]
The proof proceeds by induction.

The case $r=1$ is clear. Lemma \ref{lem3} gives
$$\begin{aligned}
B_{N,n}^{(r)}(x)&=\frac1N\left(N-\frac n{r-1}\right)B_{N,n}^{(r-1)}(x)+\frac1N\frac n{r-1}(x-(r-1))B_{N,n-1}^{(r-1)}(x) \\
&=\frac1{N^{r-1}}\left(N-\frac n{r-1}\right)\sum_{i=0}^{r-2}A_{r-1}^{(N)}(i,x;1+N(r-2)-n) \\
&\quad\times(-1)^i\binom ni i! B_{N,n-i}(x) \\
&\quad+\frac1{N^{r-1}}\frac{n}{r-1}(x-(r-1))\sum_{i=0}^{r-2}A_{r-1}^{(N)}(i,x;1+N(r-2)-(n-1)) \\
&\quad\times(-1)^i\binom {n-1}i i! B_{N,n-1-i}(x)
\\
&=\frac1{N^{r-1}}\biggl[\left(N-\frac n{r-1}\right)\sum_{i=0}^{r-2}A_{r-1}^{(N)}(i,x;1+N(r-2)-n) \\
&\quad\times(-1)^i\binom ni i! B_{N,n-i}(x) \\
&\quad+\frac{n}{r-1}(x-(r-1))\sum_{i=1}^{r-1}A_{r-1}^{(N)}(i-1,x;1+N(r-2)-(n-1)) \\
&\quad\times(-1)^{i-1}\binom {n-1}{i-1} (i-1)! B_{N,n-i}(x)\biggl]
\\
&=\frac1{N^{r-1}}\sum_{i=0}^{r-1}\biggl[\left(N-\frac n{r-1}\right)A_{r-1}^{(N)}(i,x;1+N(r-2)-n)(-1)^i\binom ni i!  \\
&\quad+\frac{n}{r-1}(x-(r-1))A_{r-1}^{(N)}(i-1,x;1+N(r-2)-(n-1)) \\
&\quad\times(-1)^{i-1}\binom {n-1}{i-1} (i-1)! \biggl]B_{N,n-i}(x)
\\
&=\frac1{N^{r-1}}\sum_{i=0}^{r-1}\biggl[\frac{N(r-1)-n}{r-1}A_{r-1}^{(N)}(i,x;1+N(r-2)-n) \\
&\quad-\frac{1}{r-1}(x-(r-1))A_{r-1}^{(N)}(i-1,x;1+N(r-2)-(n-1))\biggl] \\
&\quad\times(-1)^{i}\binom {n}{i} i!B_{N,n-i}(x) \\
&=\frac1{N^{r-1}}\sum_{i=0}^{r-1}A_r^{(N)}(i,x;1+N(r-1)-n)(-1)^i\binom ni i! B_{N,n-i}(x) \\
&\quad\text{(using $s=1+N(r-1)-n$ in (\ref{thm-eq0})).}
\end{aligned}$$
Lemma \ref{lem4} gives the result.
\end{proof}

\section{Differential equations}
\subsection{Appell polynomials}

The proof for the differential equation satisfied by the higher order hypergeometric Bernoulli polynomials
requires some basic facts on Appell polynomials. These also appear in \cite{HR}.

The Appell polynomials \cite{Ap} are defined by the generating function:
\begin{equation}\label{ap-1}
A(t)e^{xt}=\sum_{n=0}^\infty\frac{R_n(x)}{n!}t^n,
\end{equation}
where
\begin{equation}\label{ap-2}
A(t)=\sum_{n=0}^\infty\frac{R_n}{n!}t^n, \quad A(0)\neq 0
\end{equation}
is an analytic function at $t=0,$ and $R_n=R_n(0).$

A polynomial $p_n(x)$ $(n\in\mathbb N, x\in\mathbb C)$ is said to be a quasi-monomial \cite{Lu}
if there are two operators, $\hat{M},\hat{P},$ called the multiplicative and derivative operators, such that
\begin{equation}\label{m-op-def}
\hat{M}(p_n(x))=p_{n+1}(x)
\end{equation}
and
\begin{equation}\label{p-op-def}
\hat{P}(p_n(x))=np_{n-1}(x),
\end{equation}
where it is assumed (as usual) that
$$p_0(x)=1\quad\text{and}\quad p_{-1}(x)=0.$$
The operators $\hat{M}$ and $\hat{P}$ must satisfy the commutation relation
\begin{equation}\label{com}
[\hat{P},\hat{M}]=\hat{P}\hat{M}-\hat{M}\hat{P}=\hat{I},
\end{equation}
where $\hat{I}$ denotes the identity operator.

He and Ricci \cite{HR} showed that the multiplicative and derivative operators of $R_n(x)$ are
\begin{equation}\label{hr-1}
\hat M=(x+\alpha_0)+\sum_{k=0}^{n-1}\frac{\alpha_{n-k}}{(n-k)!}D_x^{n-k}
\end{equation}
and
\begin{equation}\label{hr-2}
\hat P=D_x,
\end{equation}
where $D_x=\frac{\partial}{\partial x}$ and the coefficients $\alpha_n$ are defined by
\begin{equation}\label{hr-3}
\frac{A'(t)}{A(t)}=\sum_{n=0}^\infty \alpha_n\frac{t^n}{n!}.
\end{equation}

\subsection{Proofs of Theorems  \ref{appell-dif} and \ref{appell-rec}}\label{thm-proof2}

The proofs begin with an auxiliary result.

\begin{lemma}\label{lem21}
$$\frac{{\rm d}^p}{{\rm d} x^p}  B_{N,n}^{(r)}(x)=\frac{n!}{(n-p)!}B_{N,n-p}^{(r)}(x), \quad n\geq p.$$
\end{lemma}
\begin{proof}
Start with
\begin{equation}\label{p1}
 \begin{aligned}
 \left(\frac{t^N/N!}{e^t-T_{N-1}(t)}\right)^{r} \frac{{\rm d}^p}{{\rm d} x^p} e^{xt}&= \left(\frac{t^N/N!}{e^t-T_{N-1}(t)}\right)^{r}(e^{xt}t^p)\\
 &=\left(\left(\frac{t^N/N!}{e^t-T_{N-1}(t)}\right)^{r}e^{xt}\right)t^p.
 \end{aligned}
\end{equation}
  Substituting (\ref{hn-defh}) into the right hand side of  (\ref{p1}), we have
 \begin{equation}\label{p2}
 \begin{aligned}
 \left(\frac{t^N/N!}{e^t-T_{N-1}(t)}\right)^{r}\frac{{\rm d}^p}{{\rm d} x^p}e^{xt}&=\left(\sum_{n=0}^\infty B_{N,n}^{(r)}(x)\frac{t^n}{n!}\right)t^p\\
 &=\sum_{n=0}^\infty B_{N,n}^{(r)}(x)\frac{t^{n+p}}{n!}\\&=\sum_{n=p}^{\infty} B_{N,n-p}^{(r)}(x)\frac{t^{n}}{(n-p)!}\\&=\sum_{n=p}^{\infty}\frac{B_{N,n-p}^{(r)}(x)}{(n-p)!}t^{n}.
 \end{aligned}
\end{equation}
Differentiating both sides of (\ref{hn-defh}) with respect to $x$ and using (\ref{p2}) gives
\begin{equation}\sum_{n=p}^{\infty}\frac{B_{N,n-p}^{(r)}(x)}{(n-p)!}t^{n}=\sum_{n=0}^\infty\frac{{\rm d}^p}{{\rm d} x^p}\frac{B^{(r)}_{N,n}(x)}{n!}t^n.\end{equation}
Comparing the coefficients of both sides of the above equality, produces the result.
\end{proof}

\begin{proof}[Proof of the Theorem \ref{appell-rec}]
From (\ref{hn-defh}), (\ref{ap-1}) and (\ref{ap-2}), we know that the hypergeometric Bernoulli polynomials $B_{N,n}^{(r)}(x)$ are Appell polynomials with
$$A(t)=\left(\frac{t^N/N!}{e^t-T_{N-1}(t)}\right)^r.$$
Using Lemma \ref{lem1}(ii) with $x=0$ gives
\begin{equation}\label{hy-i-1}
\begin{aligned}
\frac{A'(t)}{A(t)}&=\frac{rN}{t}-r-\frac{rN}{t}\left(\frac{t^N/N!}{e^t-T_{N-1}(t)}\right) \\
&=\frac{rN}{t}\left(1-\frac tN-\sum_{n=0}^\infty B_{N,n}\frac{t^N}{n!}\right) \\
&=\frac{rN}{t}\left(-\frac1{N(N+1)}t-\sum_{n=2}^\infty B_{N,n}\frac{t^N}{n!}\right)\\
&=r\left(-\frac1{N+1}-N\sum_{n=1}^\infty\frac{B_{N,n+1}}{n+1}\frac{t^N}{n!}\right).
\end{aligned}
\end{equation}
Here we use $B_{N,0}=1$ and $B_{N,1}=-1/(N+1).$
Then (\ref{hr-1})--(\ref{hr-3}) give the multiplicative and derivative operators of
the hypergeometric Bernoulli polynomials:
\begin{equation}\label{hr-1-1}
\hat M=\left(x-\frac r{N+1}\right)-rN\sum_{k=0}^{n-1}\frac{B_{N,n-k+1}}{(n-k+1)!}D_x^{n-k},
\end{equation}
\begin{equation}\label{hr-2-1}
\hat P=D_x.
\end{equation}
Then (\ref{hr-1-1}) gives
\begin{equation}\label{hr-2-3}\hat M B_{N,n}^{(r)}(x)=\left[\left(x-\frac r{N+1}\right)-rN\sum_{k=0}^{n-1}\frac{B_{N,n-k+1}}{(n-k+1)!}D_x^{n-k}\right]B_{N,n}^{(r)}(x).\end{equation}
(\ref{m-op-def}) with $p_{n}(x)= B_{N,n}^{(r)}(x)$ gives
$\hat M B_{N,n}^{(r)}(x)=B_{N,n+1}^{(r)}(x)$, thus by Lemma \ref{lem21}, (\ref{hr-2-3}) implies
$$B_{N,n+1}^{(r)}(x)=\left(x-\frac r{N+1}\right)B_{N,n}^{(r)}(x)-rN\sum_{k=0}^{n-1}\frac{B_{N,n-k+1}}{(n-k+1)!} \frac{n!}{(n-(n-k))!} B_{N,k}^{(r)}(x).$$
The result now follows from
$$\frac1{(n-k+1)!}\frac{n!}{(n-(n-k))!}=\frac1{n-k+1}\frac{n!}{k!(n-k)!}=\frac1{n-k+1}\binom nk.$$
\end{proof}

\begin{proof}[Proof of the Theorem \ref{appell-dif}]
Replacing $n$ by $n+1$ in Theorem \ref{appell-rec}, both sides produces
\begin{equation}\label{dif-1}
\begin{aligned}
\frac1{rN}B_{N,n}^{(r)}(x)&-\left(\frac{x}{rN}-\frac 1{N(N+1)}\right)B_{N,n-1}^{(r)}(x) \\
&+\sum_{k=0}^{n-2}\binom {n-1}k\frac{B_{N,n-k}}{n-k}B_{N,k}^{(r)}(x)=0.
\end{aligned}
\end{equation}
In Lemma \ref{lem21}, denote by $y=B_{N,n}^{(r)}(x)$, and shift indices to produce
\begin{equation}\label{dif-1-2}
B_{N,k}^{(r)}(x)=\frac{k!}{n!}y^{(n-k)}, \quad k=0,\ldots,n.
\end{equation}
This gives
\begin{equation}\label{dif-2}
\begin{aligned}
\frac1{rN}y&-\left(\frac{x}{rN}-\frac 1{N(N+1)}\right)\frac1n y' \\
&+\sum_{k=0}^{n-2}\binom {n-1}k\frac{B_{N,n-k}}{n-k}\frac{k!}{n!}y^{(n-k)}=0
\end{aligned}
\end{equation}
and writing this as
\begin{equation}\label{dif-3}
\frac1{rN}y-\frac1n\left(\frac{x}{rN}-\frac 1{N(N+1)}\right) y'+\frac1n\sum_{k=0}^{n-2}{B_{N,n-k}}\frac1{(n-k)!}y^{(n-k)}=0,
\end{equation}
produces the result.

\end{proof}

\section*{Acknowledgment}
This work  was supported by the Kyungnam University Foundation Grant, 2015.

\bibliographystyle{amsplain}

\begin{thebibliography}{10}
\bibitem{Ap}
P. Appell, \textit{Sur une classe de polynomials}, Ann. Sci. Eco. Norm. Sup. (2) \textbf{9} (1880) 119--144 (French).

\bibitem{AD} T. Agoh and K. Dilcher,
 \textit{Convolution identities and lacunary recurrences for Bernoulli numbers}, J. Number Theory \textbf{124} (2007), 105--122.



\bibitem{zeta}
A. Byrnes, L. Jiu, V.H. Moll and C. Vignat,
 \textit{Recursion rules for the hypergeometric zeta function}, arXiv:1305.1892v1 [math.NT] 8 May 2013.

\bibitem{Ch} K.-W. Chen,
\textit{Sums of products of generalized Bernoulli polynomials}, Pac. J. Math. \textbf{208} (2003) (1), 39--52.


\bibitem{DSZ}
G. Dattoli, H.M. Srivastava and K. Zhukovsky,
\textit{Orthogonality properties of the Hermite and related polynomials}, J. Comput. Appl. Math. \textbf{182} (2005), no. 1, 165--172.



\bibitem{Di}
K. Dilcher, \textit{Sums of products of Bernoulli numbers}, J. Number Theory \textbf{60} (1996), 23--41.


\bibitem{HN} A. Hassen and H.D. Nguyen,
 \textit{Hypergeometric Bernoulli polynomials and Appell sequences}, Int. J. Number Theory \textbf{4} (2008), 767--774.

\bibitem{Hassen} A. Hassen and H.D. Nguyen,
 \textit{Hypergeometric zeta functions}, Int. J. Number Theory \textbf{6} (2010), 99--126.

\bibitem{HN3} A. Hassen and H.D. Nguyen,
\textit{Moments of Hypergeometric Zeta Functions}, J. Algebra, Number Theory: Advances and Applications \textbf{7} (2012), No. 2, pp. 109--129.

\bibitem{HR}
M.X. He and P.E. Ricci,
\textit{Differential equation of Appell polynomials via the factorization method}, J. Comput. Appl. Math. \textbf{139} (2002) 231--237.

\bibitem{Ho} F.T. Howard,
 \textit{A sequence of numbers related to the exponential function}, Duke Math. J. \textbf{34} (1967), 599--615.

\bibitem{Ho2} F.T. Howard,
 \textit{Some sequences of rational numbers related to the exponential function}, Duke Math. J. \textbf{34} (1967), 701--716.

\bibitem{Ka}
K. Kamano,
\textit{Sums of products of hypergeometric Bernoulli numbers}, J. Number Theory \textbf{130} (2010), no. 10, 2259--2271.

\bibitem{KMS}
M.-S. Kim, \textit{A note on sums of products of Bernoulli numbers},
Appl. Math. Lett. \textbf{24} (2011) no. 1, 55--61.

\bibitem{KS}
M.-S. Kim and S. Hu,
\textit{Sums of products of Apostol-Bernoulli numbers}, Ramanujan J. \textbf{28} (2012), no. 1, 113--123.

\bibitem{Lu}
D.-Q. Lu, \textit{Some properties of Bernoulli polynomials and their generalizations},
Appl. Math. Lett. \textbf{24} (2011), no. 5, 746--751.

\bibitem{Na} P. Natalini and A. Bernardini, \textit{A generalization of the Bernoulli polynomials}, J. Appl. Math. \textbf{3} (2003), 155--163.

\bibitem{Nguyen}
H.D. Nguyen, \textit{Generalized binomial expansions and Bernoulli polynomials}, Integers \textbf{13} (2013) A11. 1--13.

\bibitem{NC}
H.D. Nguyen and L.G. Cheong, \textit{New convolution identities for hypergeometric Bernoulli polynomials},
J. Number Theory \textbf{137} (2014), 201--221.


\bibitem{Pe1}
A. Petojevi\'c, \textit{A note about the sums of products of Bernoulli numbers},
Novi Sad J. Math. \textbf{37} (2007), 123--128.

\bibitem{Pe}
A. Petojevi\'c, \textit{New sums of products of Bernoulli numbers},
Integral Transform Spec. Funct. \textbf{19} (2008), 105--114.


 \bibitem{Sh}
I. M. Sheffer, \textit{A differential equation for Appell polynomials},
Bull. Amer. Math. Soc. \textbf{41} (1935), 914--923.






\end{thebibliography}

\end{document}